\documentclass{article}
\usepackage[utf8]{inputenc}

\usepackage{amsmath}
\usepackage{amsfonts, amssymb}

\def\E{{\mathbb E\,}}

\def\P{{\mathbb P}}

\def\N{{\mathbb N}}

\def\SS{\mathcal S}

\def\be#1\ee{\begin{equation}#1\end{equation}}

\def\PP{{\mathcal P}}

\def\E{{\mathbb E}}

\def\MM{{\mathcal M}}
\def\N{{\mathbb N}}
\def\P{{\mathbb P}}

\newtheorem{thm}{Theorem}

\newtheorem{prop}[thm]{Proposition}

\newenvironment{proof}[1][] {\noindent {\bf Proof#1:} }{\hspace*{\fill}$\square$\medskip\par}

\title{On the maximum of random assignment process}
\author{M.A.Lifshits$^{a, *}$, A.A.Tadevosian$^{b}$}
\date{%
  {\footnotesize{
    $^a$Saint-Petersburg State University, University Emb. 7/9, 199034, St Petersburg, Russia, \texttt{mikhail@lifshits.org} \\%
    $^b$Saint-Petersburg State University, University Emb. 7/9, 199034, St Petersburg, Russia, \texttt{tadevosiaan@yandex.ru}\\
    $^*$Corresponding author
    \\[2ex]%
    }}
}
\begin{document}

\maketitle
\begin{abstract}
We describe the behavior of the expectation of the maximum for a  random assignment process  built upon a square matrix with independent entries.
Under mild assumptions on the underlying distribution, the answer is expressed in terms of its quantile function.
\end{abstract}


\section{Introduction} 

We consider the following \textit{random assignment problem}.  Let ($X_{ij}$) be an $n\times n$ random matrix with i.i.d. random entries having a common distribution $\mathcal{P}$. Let $\SS_n$ denote the group of permutations $\pi : \{1, 2, \dots, n\} \to \{1, 2, \dots, n\}$. For every $\pi\in \SS_n$ let
\[
   S(\pi)=\sum\limits_{i=1}^{n} X_{i\pi(i)}.
\]

We call such process $\{S(\pi),\, \pi \in \SS_n\}$ a \textit{random assignment process}. We are interested in the study of asymptotic behaviour of $\E \max\limits_{\pi\in \SS_n} S(\pi)$ as $n \to \infty$.


We refer to \cite{ CoppersmithSorkin,SteeleProbTheoryAndCombOpt} for many applications of assignment problem in various fields of mathematics.
\medskip

The setting with $(X_{ij})$ uniformly distributed on $[0, 1]$ was studied by Steele \cite{SteeleProbTheoryAndCombOpt} and M\'ezard and Parisi \cite{ RandomLinkMatching}, where the authors 
proved that
\[
    \E{\min\limits_{\pi\in \SS_n} S(\pi)} = \zeta(2) -\frac{\zeta(2)/2 + 2\zeta(3)}{n} + O\left(\frac{1}{n^2}\right), \text{ as } n\to \infty,
\]
$\zeta(\cdot)$ being  Riemann's zeta function. 
M\'ezard et al. \cite{SpinGlassTheory} also conjectured that in the exponential case  
when $\mathcal{P} = \mathrm{Exp}(1)$ it is true that
\[
    \E{\min\limits_{\pi\in \SS_n} S(\pi)} \to \zeta(2), \text{ as } n\to \infty.
\]
Using replica method from statistical physics \cite{DotsenkoReplica:1993}, they provided an heuristical argumentation in favor of this conjecture.

Later Parisi \cite{Parisi98aconjecture} conjectured the following explicit expression for every fixed $n$
\[
    \E{\min\limits_{\pi\in \SS_n} S(\pi)} =\sum\limits_{k=1}^{n} \frac{1}{k^{2}}.
\]
It is compatible with exact results for $n=1, 2$ and for $\zeta(2)$ limit conjecture as $n\to\infty$. This conjecture was confirmed by numeric simulations for $n=3, 4, 5$.  Special cases and generalizations are also studied in \cite{SorkinMN, BuckChan, Dotsenko_2000}.
General version of formula above was proved independently by Nair et. al \cite{NairGeneralization} and by Linusson and Wästlund \cite{LinussonGeneralization}. Later, Wästlund proposed a simpler proof \cite{WaestlundEasyProof}.

Aldous \cite{AldousZet2} gave a rigorous proof of $\zeta(2)$ limit in the exponential case. His approach is based on the assignment analysis of an infinite tree with edges provided with exponentially distributed weights, see \cite{AldousAsymp}. In the same work he showed that the distribution $\mathcal{P}$ affects $\lim_{n}\E\,{\min_{\pi} S(\pi)}$ only through the value of its probability density function at 0, thus results for $\mathcal{P} = U[0, 1]$ and $\mathcal{P} = \mathrm{Exp}(1)$ are equivalent in asymptotic sense.

In \cite{SorkinMN, BuckChan, CoppersmithSorkin, LinussonGeneralization, NairGeneralization, WaestlundEasyProof} similar problems were investigated for rectangular matrices.

The situation is very different when one deals with the variables having unbounded
distributions $\PP$. For the study of minima described above this would mean
that $\P(X_{ij}<r)>0$ for every real $r$.  In the following we find more natural
to switch from minima to maxima  and to assume that $\P(X_{ij}>r)>0$ for every real $r$. For such case the expectation of maxima does not tend anymore
to a finite limit but grows to infinity and the problem is to evaluate the order
of this growth.

In this direction, Mordant and Segers \cite{MordantSegers} showed that in Gaussian case $\mathcal{P} = \mathcal{N}(0, 1)$  it is true that 
\[
    \E{\max\limits_{\pi\in \SS_n} S(\pi)} = n \sqrt{2 \log{n}} (1 + o(1)),
\]
and 
 proposed a greedy method giving random permutation at which the permutation process attains the value asymptotically equivalent to the expectation of maximum.
 Essentially the same results for Gaussian case were obtained independently by
 the authors \cite{Lifshits2021GAP} who also used the greedy method but the priority of  \cite{MordantSegers} in its preprint version must be of course acknowledged here.  In their preprint,  Cheng et al. \cite{TTkocz} considered assignment processes on more general  graph structures and more general random
 variables satisfying some special assumptions. 
\medskip

In this paper our goal is to provide some mild and natural conditions imposed on the distribution $\mathcal{P}$ in order to obtain asymptotic expression for
expected maximum of the random assignment process. 

For notation convenience let $(X_i), X$ be some auxiliary i.i.d. random variables having the same distribution $\mathcal{P}$ and denote $M_n = \max_{1\leq i\leq n} X_i$ the maximum of i.i.d. random variables. Also let us denote $\mathcal{M}_n = \max_{\pi \in \SS_n} S(\pi)$ the maximum of the assignment process built on matrix $(X_{ij})$. Let us define the quantile function by
\[
  g(p) := \inf \{ r: {\mathbb P}(X\ge r) \le p\}, \qquad 0<p<1.
\]
 Our main result is the following theorem:

\begin{thm}\label{t:main_asymptotics} 
Assume that

1) $ \E |X|<\infty$.

2) At zero the function $g(\cdot)$ tends to infinity and is slowly varying.
\medskip

Then
\begin{equation} \label{e11}
   \E M_n = g(1/n)\, (1+o(1)), \qquad \textrm{as } n\to\infty,
\end{equation}
and
\begin{equation} \label{e12}
    \E \mathcal{M}_n = n \, g(1/n) \, (1+o(1)),  \qquad \textrm{as } n\to\infty.
\end{equation}
\end{thm}
\medskip

The structure of the paper is as follows: 
In Section \ref{s:max_indep} we study the maxima of an i.i.d. sequence
and prove \eqref{e11}. 
Then, in  Section \ref{s:MM} we  establish bilateral relations between the maximum of the assignment process and that of the corresponding i.i.d. sequence in order to derive \eqref{e12} from \eqref{e11}.
Finally, we provide expectation asymptotics of the assignment process maximum
for some standard probability distributions.


\section{Maximum's expectation in i.i.d. case}
\label{s:max_indep}

We split the proof of \eqref{e11} into two parts -- a lower bound and an upper bound.

\subsection{Lower bound}
\begin{prop} Under theorem's assumptions, it is true that
\[
   \E M_n \ge  g(1/n) \, (1+ o(1)),\quad \quad as\,\,n\to \infty.
\]
\end{prop}

\begin{proof}
For every $R>0$ we have
\begin{eqnarray} \nonumber
\E M_n &=& \int_0^\infty \P(M_n\ge r) dr -  \int_0^\infty \P(M_n\le - r) dr
\\ \label{e21}
&\ge& R \, \P(M_n\ge R) -  \int_0^\infty \P(M_n\le - r) dr.
\end{eqnarray}

Let us fix some $V > 0$ and let $R = R_n := g(V/n)$. Then
$\P(X \ge R)\ge \frac{V}{n}$ and
\begin{eqnarray*}
   \P(M_n\ge R) &=& 1-\P(M_n<R) = 1- \P(X<R)^n
\\
   &=& 1- (1-\P(X\ge R))^n
   \ge 1- \exp\left(-n \, \P(X\ge R)\right)
\\
  &\ge& 1 -   \exp\left(-n \, \tfrac{V}{n}  \right)
  =  1 -   \exp\left(- V  \right).
\end{eqnarray*}
By using slow variation property, we obtain
\[
    R \, \P(M_n\ge R) \ge g(V/n) \left(1 -   \exp\left(- V \right) \right)
    = g(1/n)\, (1+ o(1))\, \left(1 - \exp\left(- V \right) \right).
\]

On the other hand, the integrals in \eqref{e21} are uniformly bounded because
\begin{eqnarray*}
    && \int_0^\infty \P(M_n\le - r) dr \le \int_0^\infty \P(X_1\le - r) dr \\
    &\le& \int_0^\infty \P(|X_1|\ge r) dr = \E |X| < \infty,
\end{eqnarray*}
by theorem's assumption.

Now \eqref{e21} yields
\[
   \E M_n \ge  g(1/n)\, (1+ o(1)) \left(1 - \exp\left(- V \right) \right)  - O(1).
\]
Using that $g(1/n)\nearrow\infty$ and letting $V\nearrow \infty$ we obtain
\[
   \E M_n \ge  g(1/n) \, (1+ o(1)).
\]
\end{proof}
\bigskip

\subsection{Upper bound}
\begin{prop}
 Under theorem's assumptions, it is true that
\[
   \E M_n \le  g(1/n) \, (1+ o(1)),\quad \quad as\,\,n\to \infty.
\]
\end{prop}

\begin{proof}
For every $R > 0$ we have
\begin{eqnarray} \nonumber
    \E M_n &\le&   \int_0^\infty \P(M_n\ge r) \, dr \le  \int_0^R 1 \, dr + \int_R^\infty \P(M_n\ge r) \, dr
\\  \label{e31}
    &=& R \left( 1+ \int_1^\infty \P(M_n\ge Rh) \, dh\right).
\end{eqnarray}

We will now show that under the choice $R= R_n:= g(1/n)$ the integral is asymptotically negligible.

Recall the elementary inequalities
\begin{eqnarray*}
   1-p  &\ge& e^{-2p}, \qquad 0<p<1/2,
\\ 
   u    &\ge& 1- e^{-u}, \qquad u>0.
\end{eqnarray*}
By combining them, for $p\in(0,1/2)$ we have
\[
   1-(1-p)^n \le 1- e^{-2np} \le 2np.
\]
Applying this inequality to $p:=\P(X\ge r)$ we have, for all $r$ large enough,
\[
  \P(M_n\ge r) =1 -(1-\P(X\ge r))^n \le 2n\, \P(X\ge r).
\]
Therefore, for all $R$ large enough and all $n\in \N$,
\begin{equation} \label{e43}
  \int_1^\infty \P(M_n\ge Rh)\, dh
  \le 2n \int_1^\infty \P(X \ge Rh) \, dh.
\end{equation}

Next, we use the slow variation property of the function $g$.
By Karamata representation theorem \cite{Bing}, for every  fixed $\delta>0$ and
$V>0$ we have, for all $\rho$ small enough and all $\ell\le 1$
\[
   \frac{g(\ell \rho)}{g(\rho)} < (1+\delta) \ell^{-1/V}.
\]
Let $h\ge 1+\delta$. Then choose $\ell:= \left(\frac{h}{1+\delta}\right)^{-V}\le 1$
which satisfies
\[
   (1+\delta) \, \ell^{-1/V}= (1+\delta) \, \frac{h}{1+\delta} = h.
\]
We obtain $g(\ell \rho) < h \, g(\rho)$. It follows that
\[
  \P\left( X\ge g(\rho)h \right) \le \P\left( X > g(\ell \rho) \right) \le \ell \rho
  =\left(\frac{h}{1+\delta}\right)^{-V} \rho.
\]
Integration  of this bound yields for $V>1$
\begin{eqnarray*} 
  \int_{1+\delta}^\infty  \P\left( X\ge g(\rho) h\right) \, dh
  &\le& \rho  \int_{1+\delta}^\infty \left(\frac{h}{1+\delta}\right)^{-V}\, dh
\\  
  &=& \rho \, (1+\delta) \int_1^\infty \tilde{h}^{-V}d\tilde{h}
  = \frac{\rho(1+\delta)}{V-1}.
\end{eqnarray*}
We choose here $\rho=\rho_n:=1/n$ and $R=R_n:=g(\rho_n)=g(1/n)$
to get
\[
  n \int_{1+\delta}^{\infty} \P(X\ge  R_n h)\, dh \le n \cdot \frac{1+\delta}{n(V-1)}
  =\frac{1+\delta}{V-1}.
\]
for all $n$ large enough. By letting $V\nearrow \infty$, we have
\[
     n \int_{1+\delta}^{\infty} \P(X\ge  R_n h) \, dh = o(1), \qquad \textrm{as } n\to \infty,
\]
for every fixed $\delta$.

Using $R_n=g(1/n)$ we also have a trivial bound
\[
     n \int_1^{1+\delta} \P(X\ge  R_n h)\, dh \le n \, \delta \, \P(X> R_n) \le n\, \delta\, \frac1n =\delta.
\]
By summing up the bounds for two integrals, we obtain
\[
    n \int_{1}^{\infty} \P(X\ge R_n h) \, dh \le \delta + o(1),
\]
and letting $\delta\searrow 0$ we have
\[
    n \int_{1}^{\infty} \P(X\ge  R_n h)\, dh = o(1).
\]
By combining this estimate with \eqref{e43} and \eqref{e31} we finally get
\[
   \E M_n \le R_n (1+o(1)) = g(1/n) (1+o(1)).
\]
\end{proof}


\section{Relations between $M_n$ and $\MM_n$}
\label{s:MM}

The goal of this section is to connect the expectations of  $M_n$ and $\MM_n$
and to show that under theorem's assumptions \eqref{e11} implies \eqref{e12}.

We start with a trivial estimate
\[
  \MM_n\le \sum_{i=1}^{n} \max_{1\le j\le n} X_{ij} := \sum_{i=1}^{n} M_n^{(i)},
\]
where $M_n^{(i)}$ are independent copies of $M_n$. Hence, using \eqref{e11},
\begin{equation} \label{MM1}
   \E \MM_n\le  \sum_{i=1}^{n}\E M_n^{(i)} = n\, \E M_n = n\, g(1/n) \, (1+o(1)).
\end{equation}

In order to get an opposite bound, we have to recall the "greedy method" of constructing a quasi-optimal permutation $\pi^*$ providing sufficiently large value of the assignment process. In Gaussian context,  this method was introduced in \cite{MordantSegers} and later in \cite{Lifshits2021GAP}. 

Let 
$[i] := \{1, 2, \dots, i\}$.
Define
\[
    \pi^*(1) := \arg \max\limits_{j\in[n]} X_{1j},
\]
and let for all  $i=2,\dots, n$
\[
    \pi^*(i) := \arg \max\limits_{j \not\in \pi^*([i-1])} X_{ij}.
\]

It is natural to call this strategy greedy, because at every step we consider the row $i$, take the maximum of its available elements (without considering the influence of this choice on subsequent steps) and then forget the row $i$ and the corresponding column  $\pi^*(i)$.

We stress that at every step we make choice from i.i.d. random variables
with distribution $\PP$. They are also independent on the variables used at other steps but we don't need this fact here. The number of variables used at consequent steps is decreasing from $n$ to $1$.

By using the greedy method, we have
\[
   \MM_n \ge \sum_{i=1}^{n} X_{i\,\pi^*(i)}
   =  \sum_{i=1}^{n} \max\limits_{j \not\in \pi^*([i-1])} X_{ij}
   :=  \sum_{i=1}^{n} L_i,
\]
where $L_i$ is equidistributed with $M_{n-i+1}$.

Hence, for every $\delta\in (0,1)$
\begin{eqnarray*}
   \E \MM_n &\ge& \sum_{i=1}^{n} \E L_i  = \sum_{i=1}^{n} \E M_i
\\
     &\ge& (n-[\delta n])\ \E M_{[\delta n]} + \delta n \, \E M_1
\\
  &=& (1-\delta) n\, g(1/[\delta n])\, (1+o(1)) + O(n)
\\
  &=&  (1-\delta) n  \, g(1/n)\, (1+o(1)),
\end{eqnarray*}
where we consequently used that the sequence $(\E M_n)$ is non-decreasing,
the asymptotics \eqref{e11}, the slow variation property of the function $g$,
and the fact that  $g(1/n)\nearrow\infty$ as $n\to \infty$.

Since $\delta$ can be chosen arbitrarily small, we obtain
\begin{equation} \label{MM2}
  \E \MM_n \ge n \ g(1/n) \, (1+o(1)).
\end{equation}

Now inequalities \eqref{MM1} and \eqref{MM2} prove that \eqref{e11} implies
\eqref{e12}. 


\section{Examples}
In this section we provide the main asymptotic terms of $\E M_n$ and $\E \mathcal{M}_n$ as $n \to \infty$ for some standard distributions satisfying assumptions of Theorem \ref{t:main_asymptotics}.

\begin{center}
\begin{tabular}{ |c|c|c|c| } 
\hline
Distribution $\mathcal{P}$ & Tail quantile function $g(p)$ & $\E M_n$ & $\E \mathcal{M}_n$ \\
    \hline
    $\mathcal{N}\, (0, 1)$ & $\sqrt{2}\,\mathrm{erfc}^{-1} (2p)$ & $\sqrt{2 \log{n}}$ & $n\sqrt{2 \log{n}}$ \\ 
    $\mathrm{Exp}\,(1)$ & $-\log{p}$ & $\log{n}$ & $n\log{n}$\\ 
    $\mathrm{Gumbel}\,(0, 1)$ & $-\log{\log{\frac{1}{1-p}}}$ & $\log{n}$ & $n\log{n}$\\ 
    $\mathrm{Laplace}\,(0, 1)$ & $ -\log{(2p) },\, 0 < p < 1/2$ & $\log{n}$ & $n\log{n}$ \\
   
   $\mathrm{Poisson}\,(\lambda)$ & $
   -\log{p}/\log{\log{(1/p)}}$ & $\log{n}/\log{\log{n}} $ & $n \log{n} / \log{\log{n}}$ \\
    \hline
\end{tabular}
\end{center}
\section*{Acknowledgement} 
We are grateful to Professor Th.\,Tkotz for sending us preprint \cite{TTkocz}
and for other useful advise.

The work supported by Russian Science Foundation Grant [21-11-00047].
 

\end{document}